\documentclass{amsart}
\usepackage{amsfonts,amscd, amssymb}
\usepackage{tikz-cd}
\usepackage{array}
\newtheorem{theorem}{Theorem}[section]
\newtheorem{lemma}[theorem]{Lemma}
\newtheorem{corollary}[theorem]{Corollary}
\newtheorem{proposition}[theorem]{Proposition}

\theoremstyle{remark}

\theoremstyle{definition}

\newtheorem{example}[theorem]{Example}
\numberwithin{equation}{section}
\makeatother

\DeclareMathOperator{\bC}{{\mathbb C}}

\DeclareMathOperator{\bN}{{\mathbb N}}

\DeclareMathOperator{\Cdb}{{\mathbb C}}

\DeclareMathOperator{\Fdb}{{\mathbb F}}

\newcommand*{\tran}{\mathsf{T}}

\DeclareMathOperator{\minten}{\otimes_{\rm min}}

\DeclareMathOperator{\cM}{{\mathcal M}}

\DeclareMathOperator{\cL}{{\mathcal L}}

\usepackage{tikz}

\usepackage{pst-node}
\usepackage{tikz-cd} 

\usepackage[hidelinks]{hyperref}

\title{Commutativity of operator algebras}

\author{David P. Blecher}
	
\address{Department of Mathematics, University of Houston, Houston, TX 77204-3008.}
	
\email{dpbleche@central.uh.edu} 

\date{November 17, 2025.  To appear Expositiones Mathematicae}

\begin{document} 


\begin{abstract} 
We call an operator algebra $A$ {\em reversible} if $A$ with reversed multiplication  is also an abstract operator algebra (in the modern operator space sense).
This class of operator algebras is intimately related  to  the {\em symmetric operator algebras}:
the subalgebras of $B(H)$ on which the transpose map is a  complete isometry.  In previous work we 
studied the unital case, where reversibility is equivalent to commutativity.  We give many sufficient conditions
under which a nonunital reversible or symmetric  operator algebra is commutative.
We also give many complementary results of independent interest, and solve a few open questions from previous papers.
Not every reversible or symmetric operator algebra is commutative, however we show that they all are 3-commutative.  That is, order does not matter in  the product of three or more  elements from $A$.  The proof of this relies on a  technical analysis involving the injective envelope.  Indeed nonunital algebras are often enormously more complicated than unital ones 
in regard to the topics we consider. On the positive side,  our considerations raise 
very many questions even for low dimensional matrix algebras, some of which are of a computational nature and might be suitable for undergraduate research.  The canonical anticommutation relations from mathematical physics play a significant role.
  \end{abstract}

\maketitle

\begin{center} {\em Dedicated with affection to David Larson, who while usually nonselfadjoint, was invariably positive} \end{center}

\section{Introduction} 

If $A$ is a matrix algebra (isomorphic to an 
algebra of finite matrices), then $A$ with reversed multiplication is also 
(representable as) a matrix algebra.
This is because such matrix algebras are exactly the finite dimensional algebras.   (Indeed any 
finite dimensional algebra $A$ may be viewed as a subalgebra of $\cL(V) \cong M_n$ where $V = A^1$, the unitization, and $n = {\rm dim}(A^1)$.)   
The situation is startlingly different for operator algebras in the modern operator space sense.
We define a concrete  operator algebra to be an 
algebra $A$ of bounded operators on a Hilbert space $H$,
together with the canonical matrix norms (the norm on $M_n(A)$ 
inherited from $M_n(B(H)) \cong B(H^{(n)})$).  An {\em abstract operator algebra}  is 
a Banach algebra $A$ with a norm on $M_n(A)$ for each $n \in \bN$, which is   isomorphic to a concrete  operator algebra via 
a completely isometric (this is defined below) surjective homomorphism.    We say that an abstract operator algebra  is {\em reversible} if $A$ with reversed multiplication (but with same matrix norms) is also an abstract operator algebra in this sense.  The reversible operator algebras  are intimately related  to  the {\em symmetric operator algebras}:\ 
namely the subalgebras of $B(H)$ on which the transpose map is a  complete isometry. Every symmetric operator algebra is reversible.  Conversely to every reversible operator algebra we can associate a canonical symmetric operator algebra $B$ which it is isometrically isomorphic to.  Of course $B$ is commutative if and only if $A$ is.
We proved in 
\cite{BComm} that a unital operator algebra is reversible if and only if it is commutative.  This is also true for $C^*$-algebras,
and more generally for operator algebras with a contractive approximate identity (see e.g.\ 2.2.8 in 
 \cite{BLM}).  However it turns out that it is not quite true for all 
nonunital operator algebras.  In this paper we give several sufficient conditions for which it is true. 
This solves questions posed in \cite{BComm}.  We also solve a few other open questions from previous papers.    We also show that every reversible operator algebra is 3-commuting: the order of the product of three or more  elements from $A$ does not matter.  Although the above statements are easy to understand, our proofs are often rather technical, e.g.\ often relying on   technical analysis involving the injective envelope  and TRO (ternary rings of operators) methods, 
 and the Kaneda-Paulsen characterization 
of nonunital operator algebras in terms of the injective envelope.   

We are not aware of any work on 3-commutativity in the literature.    Similarly, since there seems to be no neat completely general criterion for reversibility to imply commutativity, 
our considerations raise 
very many questions even for low dimensional matrix algebras.   Many of these would be suitable for undergraduate research projects, for example involving computations with  $3 \times 3$ and $4 \times 4$ matrices, or structure theory of finite dimensional algebras.   Reversible (and more generally, 3-commuting) operator algebras turn out to be   triangularizable operator algebras in the sense of  e.g.\ Radjavi and Rosenthal \cite{RR} or \cite{KS}.
Thus there are good links to the theory of triangular operator algebras, and to work of Larson \cite{L} and others on similarity, which raise interesting questions.  For example,  one can try to classify subclasses of these in low dimensions using Burnside's theorem and results such as (the nonunital variant of) Theorem 1.5.1 of \cite{RR}.  

 Turning to the structure of our paper, in Section
        \ref{prel} we discuss a few preliminaries and definitions.  In Section
        \ref{rsy} we prove a couple of  
        sufficient conditions for reversibility to imply commutativity.  Similarly we discuss symmetric operator algebras in the spirit of \cite[Section 2]{BWinv}.
        We also take the opportunity to correct a couple of misstatements from \cite{BWinv}. In Section
        \ref{sp} we show that operator algebras may be represented in a  certain `standard position' where the TRO structure of the injective envelope is compatible with the algebra structure.  We use this to associate a canonical commutative operator algebra with any reversible operator algebra $A$, and show that $A$ is 3-commutative in the sense of the abstract.  The main theorem here is used as a tool to prove commutativity in other sections of our paper.  In Section
        \ref{kegs} we give some examples based on the canonical anticommutation relations (CAR) from mathematical physics.  
        These supply counterexamples to open questions in \cite{BComm} and other papers of ours.  In Section
        \ref{sess} we study an interesting class of operator algebras,
        those which are subalgebras of their injective envelope when the latter is a $C^*$-algebra.  This turns out to coincide with the class of operator algebras for which the unitization $A^1$  of $A$  is an essential operator space extension of $A$, or equivalently that $I(A) = I(A^1)$.  We give a low dimensional example of 
        an algebra $A$ for which  $A^1$ is  a rigid  but not essential operator space extension of $A$. 
        This example is also a symmetric noncommutative operator algebra. 
        In the short Section 
    \ref{joa} we consider Jordan operator algebras, answering  several open questions. 
        Finally in Section
        \ref{last} we give a few more  
        sufficient conditions under which a nonunital reversible or symmetric  operator algebra is commutative.  
        We also show that a  finite dimensional 
     reversible  operator  algebra is the algebra direct sum of two ideals $C$ and $K$ where $C$ is a unital commutative operator  algebra (or $(0)$), and $K$ is a  nilpotent ideal and is also reversible, 
     and $CK = KC = (0)$.  This reduces the study of finite dimensional 
     reversible  operator  algebras to the case of nilpotent algebras.  Indeed for matrix algebras,
     to the case of strictly upper triangular matrix algebras.
     
 \section{Preliminaries} \label{prel} 
 
 We write $M_n$ (resp.\ $M_{m,n}$) for the  $n \times n$ (resp.\ $m \times n$) matrices.  
 The letters $H, K$ are usually reserved for Hilbert spaces.  
 For subsets $X, Y$ of an algebra $A$ we write $XY$ for the span of products $xy$ for $x \in X, y \in Y$.  The reader will need to be familiar with the  basics of  operator spaces and operator algebras,  
as may be found in  e.g.\ \cite{BLM,Pau,P,Dav}.  We refer to these 
 for more information on the topics below  (particularly \cite{BLM}, since it covers all the topics).  
 Nearly all of our proofs  
 seem to all work in both the real and complex case, and we shall 
 usually say nothing about the underlying field after early sections (readers needing more detail on the real case of our preliminaries are referred e.g.\ to \cite{BReal}). 
A  normed algebra $A$  is {\em unital} if it has an identity $1$ of norm $1$, 
and a map $T$ 
is unital if $T(1) = 1$.  
A {\em concrete complex (resp.\ real) operator space} $X$ is a complex (resp.\ real)  subspace of $B(H)$ for $H$ a complex (resp.\ real) Hilbert space. For $n \in \bN$ we have the identification $M_n(B(H)) \cong B(H^{(n)})$ where $H^{(n)}$ is the $n$-fold direct sum of $H$. From this identification, $M_n(X)$ inherits a norm. 
If $T : X \to Y$ we write $T^{(n)}$ for the canonical `entrywise' amplification taking $M_n(X)$ to $M_n(Y)$.   
The completely bounded norm is $\| T \|_{\rm cb} = \sup_n \, \| T^{(n)} \|$, and $T$ is 
completely  contractive (resp.\
completely isometric) if  $\| T \|_{\rm cb}  \leq 1$ (resp.\ $T^{(n)}$ is isometric for each $n$). 

        We will use very often Hamana's theory of the injective envelope $I(X)$ of an operator space, 
        which may be found e.g.\ in \cite[Chapter 4]{BLM} or \cite{Ha}, or in the several other basic texts on operator spaces. 
        The real case may be found in e.g.\ \cite{BReal,BCK}. 
        We will also use the theory of TRO's (ternary rings of operators) which may be found in e.g.\ \cite[Chapters 4 and 8]{BLM}.  A subTRO $Z$ of a $C^*$-algebra $A$, or 
        of a TRO in $A$, is a linear subspace with $Z Z^* Z \subseteq Z$. 
        The
        injective envelope $Z = I(X)$ like every injective operator  space has a unique structure as a TRO.  Thus we may write unambiguously 
        $x y^* z$ for $x, y, z \in Z$. 
        For some of the computations in our paper, or in regard to undergraduate research, we remark that for low dimensional examples it is often easy to compute the injective envelope $Z$ of 
        a subspace $A$ of $M_n$, which coincides with Hamana's ternary envelope (see e.g.\
        \cite[Chapters 4 and 8]{BLM})  in this case.   
        Indeed $Z$ is finite dimensional, and  
        (in very low dimensional examples) will         often  be the 
        TRO $W$ in $M_n$ generated by $A$.  That is, $W$ is the span of products $a_1 a_2^* a_3 a_4^* \cdots 
                a_m$ for $a_k \in A.$ 
        Alternatively, $W = B A$ where 
        $B$ is the $*$-algebra generated by 
        products $a_1 a_2^*$ for $a_k \in A.$
        Even in cases where $Z \neq W$ one may often work in $W$, or one may obtain $Z$ from $W$ as a quotient, by deleting blocks in the `block rectangular matrix form' of $W$.  We explain this a little further (see also the Notes to Section 8.3 in \cite{BLM}).  Recall the well known fact 
        that any finite dimensional complex TRO $W$ 
        is unitarily equivalent to $\oplus_{k = 1}^r \, M_{n_k,m_k}$ for some $m_k,n_k \in \bN$.  
        (This may be deduced by applying the matching fact about finite dimensional $C^*$-algebras to the `linking
        $C^*$-algebra of $W$.) 
        The $M_{n_k,m_k}$ are the `blocks' referred to above.
        Some of these blocks are needed in order to contain $A$ completely isometrically.
        One obtains $Z$ from $W$ by deleting the blocks not needed to `support'  
        $A$ in this way.
        If $W$ is `simple' (that is, $W W^*$ has no nontrivial ideals, or equivalently no nontrivial  central projections), then $r = 1$ here.  Thus in the         discussion above if $W$ is simple then $W = Z = I(A)$.  For example, the injective envelope of the strictly upper triangular matrices in $M_n$ is easily seen by the above considerations to be the copy of $M_{n-1}$ in the upper right  corner of $M_n$.

        By Meyers theorem any operator algebra $A$ has a unitization $A^1$, which is unique up to 
        completely isometric isomorphism.  If $A$ has the trivial (zero) product the operator  space structure
        on $A^1$ is the one given by 
        $$\begin{bmatrix} \lambda I_H & x \\ 0 & \lambda I_H \end{bmatrix} \in M_2(B(H)) = B(H^{(2)})$$ for scalars $
       \lambda$ and $x \in A$, if $A \subseteq B(H)$ as an operator space.  One may compute the norm of this (on $A^1 = M_1(A^1)$) explicitly, indeed it is the norm of $$\begin{bmatrix} |\lambda | & \| x \| \\ 0 & | \lambda |\end{bmatrix},$$ 
       and this number has an explicit formula \cite[Lemma 2.3]{BNm2}. 
   However for general nontrivial   operator algebras there is no such convenient formula.    Indeed the 
   operator  space structure
        on $A^1$ can be slippery (see e.g.\ Corollary \ref{uni}).
   
   Unital operator algebras were characterized abstractly in \cite{BRS}.   We recall the Kaneda-Paulsen characterization of nonunital operator algebras:
        
        \begin{theorem} \label{KP} Operator algebra products on an operator space $X$ are in bijective correspondence 
        with elements $z \in {\rm Ball}(I(X))$ for which $X z^* X \subseteq X$.  For such $z$ the associated operator algebra product on $X$ is $x z^* y$. \end{theorem}
        
        \begin{proof} A short proof may be found in 
        \cite[Theorem 5.2]{BNm2} (the original article is \cite{KP}). 
        Indeed that $x z^* y$
        is an operator algebra product  is immediate from 
         Remark 2 on p.\ 194 of \cite{BRS}, viewing $I(X)$ as a TRO in $B(H)$, and setting $V$ in that remark to be $z^*$. The real case is in \cite[Section 3]{BReal}. \end{proof}

More generally than one direction of this, we have:

 \begin{theorem} \label{BRSKP}  ({\rm C.f.\ } \cite[Remark 2 on p.\ 194]{BRS}) \ If $X$ is a subspace of a TRO or $C^*$-algebra $B$, and if we have a fixed element 
$z \in {\rm Ball}(B)$ with $X z^* X \subseteq X$, then $X$ is an operator algebra with product $x z^* y$. 
%gg
Conversely if  $A$ is a subalgebra of $M_n$ for $n \in \bN$ and if $W$ is the TRO generated by $A$ in $M_n$ then 
there exists $z \in {\rm Ball}(W)$ with $xy = x z^* y$ for $x, y \in A$.  \end{theorem}

%gg
\begin{proof}  The first assertion follows by the same
proof from~\cite{BRS} alluded to in the last lines of the proof
of  \ref{KP}.  To see the `conversely',  note that 
$W$ is injective (as is any finite dimensional TRO), and let $P : M_n \to W$ be a completely contractive projection.  As in the 
proof of  \cite[Theorem 5.2]{BNm2} if $z = P(1) \in W$ then $xy = x z^* y$ for $x, y \in A$. \end{proof}
        
      In places in our paper we propose using  these last results
       as an often  effective way to test if the product on        an algebra is an  operator algebra product.
       That is, finding such an element $z$ as above with the product $xy = xz^*y$ for all $x,y \in A$,
       or showing no such $z$ exists.
       We have already said that in many low dimensional examples even finding $I(A)$ 
       can be  practical, and then one may seek for $z$ in here.  As an example, the reader could 
       try this with $A$ the strictly upper triangular matrices in $M_3$, and with the reversed product. 

  We recall for any operator space $X$ the {\em opposite}, {\em adjoint}, and {\em conjugate} operator spaces $X^\circ, \bar{X},$ and $X^\star$ from e.g.\ 1.2.25 in \cite{BLM} and \cite{P}.  
Here $X^\circ$ is $X$ but with `transposed matrix norms' $|||[x_{ij} ]||| = \| [x_{ji} ] \|$.
Similarly $X^\star$ is the set of formal symbols $x^\star$ for $x \in X$, but with  the same operator space structure 
as  $\{ x^* \in B : x \in X \}$, if $X$ is (completely isometrically) a subspace of a 
$C^*$-algebra $B$.   Thus $\| [x_{ij}^\star] \|_n = \| [x_{ji} ] \|_n$, so that $X^\circ = X^\star$ real completely isometrically via the map $x^\circ \mapsto x^*$. 
On the other hand $\bar{X} = (X^{\star})^\circ$, with $\| [\overline{x_{ij}}] \|_n = \| [x_{ij}] \|_n$ (see \cite[Section 5]{BReal}). 
We define  ${\rm Sym}(X)$ to be $X$ but with
matrix norms $\max \{ \| [x_{ij} ] \| , \| [x_{ji} ] \| \}$.  Equivalently the map $x \mapsto (x,x^{\circ}) : {\rm Sym}(X) \to X \oplus^\infty X^{\circ}$ is a complete isometry.

We say that an operator space  $X$ is {\em symmetric} 
if  $\| [a_{ij} ] \| = \| [a_{ji}] \|$ for $a_{ij} \in X$.   Equivalently, $X = {\rm Sym}(X)$; or $X \cong X^{\circ}$ linearly completely isometrically via the map $a \mapsto a^\circ$.   If $X \subseteq B(H)$ this is also equivalent to saying that the transpose map on $B(H)$  with respect to some orthonormal basis of $H$ 
 (which is an isometric antihomomorphism), is 
 completely isometric (or even just completely contractive).
 
This definition has the subtlety that it is not in general equivalent to saying that $X \cong X^{\circ}$ linearly completely isometrically, or even 
algebraically completely isometrically if $X$ is an operator algebra.   For example if $A = M_n$ then  $A \cong A^{\circ}$ even $*$-isomorphically, but $A$ is not symmetric.

We recall that Ruan showed that every real  operator space $X$ has a unique `completely reasonable' complexification (see e.g.\ \cite{BReal,BCK}, in particular a short proof may be found in 
Theorem 2.2 of the latter reference).  If $X \subset B(H)$ then we can identify $X_c$ with $X + iX$ in $B(H)_c \cong B(H_c)$.  Ruan also gave a pretty abstract characterization of $X_c$ inside $M_2(X)$ (see the cited references). 

\begin{proposition} \label{ciop}  If $X$ is a real  operator space and $X_c$ its complexification
 then
\begin{itemize}
\item [(1)]  $(X_c)^\circ = (X^\circ)_c$ and ${\rm Sym}(X_c) = ({\rm Sym}(X))_c$.
\item [(2)] $(X_c)^\star= (X^\star)_c$ and $\overline{X_c} = (\bar{X})_c$.
\end{itemize}
\end{proposition}

\begin{proof}       We leave most of this as an exercise.   Some use Ruan's results
\cite[Proposition 2.1 or Theorem 2.2]{BCK}. The fact  that 
 $\bar{X} = (X^{\star})^\circ$ is helpful for some of these. Since ${\rm Sym}(X) \subseteq X \oplus^\infty X^{\circ}$ we have
 ${\rm Sym}(X)_c \subseteq X_c \oplus^\infty X^{\circ}_c$, from which  it is easy to see that 
 ${\rm Sym}(X_c) = ({\rm Sym}(X))_c$.
\end{proof}

The following simple result is as in \cite{BComm,BReal}: 

\begin{proposition} \label{p1}   \ If $A$ is a real  or complex   operator algebra then $A^{\circ}$ is an operator algebra with reversed multiplication.
 \end{proposition} 

Thus when we refer to $A^{\circ}$ as an  algebra it will usually be with the above operator space structure and with reversed multiplication, unless noted to the contrary.  We sometimes write
$A^{\rm rev}$ for $A$ with its usual  operator space structure and with reversed multiplication.

The spaces $X^\circ$ and $X^\star$ play a significant role in \cite{BWinv}. There are however a few misstatements in that paper, which we take the opportunity to correct.   Most of them are around  \cite[Corollary 2.5]{BWinv}.  First, above that Corollary it is stated that symmetric operator algebras ``were introduced in \cite{BComm} where it was observed that such algebras were commutative''.   As noted above, \cite{BComm}  only proves commutativity in the unital case.  
 Second, the statement of \cite[Corollary 2.5]{BWinv} should begin with ``An  operator algebra $A$  is symmetric {\em and commutative} 
 if and only if there exists...''.    This misstatement also appears in the middle of the first paragraph of Section 1.1 of the paper, the word {\em commutative} should be added before 
 ``symmetric operator algebras."  General symmetric operator algebras are treated in the present paper. 
 Similarly the last sentence of the statement of \cite[Corollary 2.5]{BWinv} should begin with ``Equivalently, $A$  is symmetric {\em and commutative}  if and only if there exists...''.   Indeed the first line of the proof only works if $A$ is also commutative, but the proof still seems to be relatively deep even in this case.  We state the result formally as Corollary \ref{chsym} below.  

\begin{theorem} \label{thmsym}  A real  (resp.\ complex)   operator space $X$ is  symmetric if and only if $X$ is linearly completely isometrically isomorphic to a 
real  (resp.\ complex)  space of operators  on a complex Hilbert space 
whose matrix (representation) with respect to some fixed orthonormal basis is 
symmetric (that is the matrices are symmetric in the undergraduate sense).   Equivalently, if and only if  there exists a conjugation $c$ on a complex 
 Hilbert space $H$ on which $X$ may be real  (resp.\ complex)  
 completely isometrically represented such that $c x c = x^*$ for all $x \in X$ (here we are identifying $X$ with its image in $B(H)$). 
\end{theorem}

 \begin{proof}   If $X$ is a symmetric operator space then $X$ is a symmetric commutative 
 operator algebra with the zero product. 
 By the corrected statement of \cite[Corollary 2.5]{BWinv} above, or by Corollary \ref{comm} below, $X$  is completely isometrically isomorphic to a
space of  symmetric matrices with the zero product, and there exists a conjugation $c$ with the stated 
properties.   Conversely if $c$ is a conjugation as above then 
$$\| [a_{ij} ] \| = \| [a_{ji}^* ] \| = \| [ ca_{ji} c ] \| = \| [a_{ji} ] \|$$ 
for $[a_{ij}] \in M_n(X)$.  So $X$ is symmetric. 
If $X$ is a space of symmetric matrices then it is easy to see that it is a symmetric operator space. 
 The rest is as in and around  \cite[Corollary 2.5]{BWinv}. 
\end{proof}

\begin{proposition} \label{classof} The class of reversible (resp.\ symmetric) operator algebras is closed under 
subalgebras, quotients by closed ideals, $\oplus^\infty$ direct sums, biduals, and minimal tensor product.  A real operator algebra $A$ is  reversible (resp.\ symmetric)
if and only if its complexification $A_c$ 
(see e.g.\ \cite{BReal,BCK}) is  reversible (resp.\ symmetric). 
\end{proposition} 

\begin{proof}   We leave some of these as exercises.
If an operator algebra $A$ is reversible (resp.\ symmetric) and $I$ is a closed ideal, then $I$ is reversible
(resp.\ symmetric).  Moreover $(A/I)^\circ = A^\circ/I^\circ$ as operator spaces.  So if $A$ is symmetric then so is 
$A/I$.   If $A$ is reversible, then $I^{\rm rev}$ is an ideal in $A^{\rm rev}$ and $A^{\rm rev}/I^{\rm rev}$ is 
an operator algebra, and is completely isometrically isomorphic 
to $(A/I)^{\rm rev}$ as algebras.  So the latter is an operator algebra and 
 so $A/I$ is reversible.

Suppose that $\pi : A^{\rm rev} \to B(H)$ is a completely isometric homomorphism.  Taking $\pi^{**}$, and using Arens regularity of $A$ and the fact that $(A^{\rm rev})^{**} = (A^{**})^{\rm rev}$ as Banach algebras and as operator spaces we see that $A^{**}$ is reversible.  
Similarly for symmetric since $(A^{\circ})^{**} \cong (A^{**})^{\circ}$. 

 If $A, B$ are  reversible then 
$A^{\rm rev} \minten B^{\rm rev} \cong (A \minten B)^{\rm rev}$ completely isometrically  and 
as algebras.   So $A \minten B$ is reversible. Similarly for symmetric since $A^{\circ} \minten B^{\circ} \cong (A \minten B)^{\circ}$. 

If $\pi : (A,\circ) \to B(K)$ is a completely isometric real
 homomorphism for the reversed multiplication $\circ$,
then $\pi_c : (A,\circ)_c = (A_c,\circ) \to B(K)_c$ is a completely isometric  homomorphism for the reversed multiplication.    So $A_c$ is 
reversible (resp.\ symmetric) if $A$ is (for symmetric use the relation $(A_c)^\circ = (A^\circ)_c$
from Proposition \ref{p1}).    
The converses follow since $A$ is a subalgebra of $A_c$. 
\end{proof}

{\bf Remark.}  The class of reversible (resp.\ symmetric) operator algebras is also closed under the `opposite', `adjoint', `conjugate'
constructions mentioned earlier.  Note that ${\mathcal U}(X)$, for an operator space $X$, is only reversible if 
$X= (0)$.  Indeed this is unital and so is only commutative in this case.  Similarly if $A$ is commutative we obviously  cannot expect $M_n(A)$ to be reversible in general
unless $n = 1$.    The algebra generated by a symmetric operator space need not be symmetric.  This is false even for 
symmetric operator systems.  One way to see this is to take any operator system $V$ which is
not `minimal', that is, not a function system, such as $V = M_2$.    Then ${\rm Sym}(V)$ is a
symmetric operator system but is also not `minimal'.  The closed algebra generated by a symmetric operator system is a $C^*$-algebra, which would be commutative if it were symmetric.   However if $C^*({\rm Sym}(V))$ was commutative then 
${\rm Sym}(V)$ is `minimal', a contradiction.  This example was constructed in a discussion with  M. Kalantar.

\section{Reversible and symmetric operator algebras} \label{rsy}

The complex version of 
most of the following results were proved in the unital case in \cite{BComm}.  Whether 
(i)--(iv) were equivalent for general nonunital operator algebras 
was  stated there to be an open question in the nonunital case.  
 The approximately unital complex case is not hard, following similarly to the unital case, and has been stated in some form in some of our referenced papers.  We solve the general nonunital operator algebra case of these  questions in the present paper, showing that they are not true in general, but giving many sufficient conditions for commutativity. 

We say that an operator algebra $A$ is {\em idempotent} if $A^2$, the closure of the span of products of two elements is dense in $A$.   This includes all operator algebras with a one-sided or two-sided bounded approximate identity, topologically simple operator algebras,  finite dimensional semisimple operator algebras,   etc.
We say that an operator algebra $A$ is {\em left faithful} if $aA = 0$  implies that $a =0$.
Similarly for {\em right faithful}: if $Aa = 0$ then  $a =0$, and {\em faithful} means both left and right faithful.
We shall at some point in the following result appeal to Theorem \ref{both2} from the next section.  We do this to avoid for now the rather technical details necessary for that result. 

 \begin{theorem}  \label{mainc} For an
 idempotent 
 or left or right faithful real  (resp.\ complex)  operator algebra $A$ the following are equivalent:
  \begin{itemize}
\item[(i)] $A$ is a commutative algebra. 
 \item[(ii)] $A$  is reversible.
  \item[(iii)] $A^{\circ}$ with usual multiplication of $A$ is (completely isometrically) an operator algebra.
  \item[(iv)] {\rm Sym}$(A)$ is (completely isometrically) an operator algebra with usual multiplication of $A$.
 \item[(v)]  $A$ is  isometrically isomorphic to a real  (resp.\ complex) algebra of 
 operators on a complex Hilbert space whose matrix (representation) with respect to some fixed orthonormal basis is 
symmetric  (that is the matrices are symmetric in the undergraduate sense).   \end{itemize} 
  \end{theorem}

  \begin{proof}  (i) $\Rightarrow$ (ii)\ This is obvious, and holds without the `idempotent'  or faithful assumption.   
  
  (ii) $\Rightarrow$ (i)\ If $A$  is reversible
  then it is 3-commutative by Theorem \ref{both2} (4). 
  So if $A$ is idempotent then (i) holds.
  Similarly, if $A$ is left faithful then (i) holds by Theorem \ref{both2} (2) and the first Remark  after that result.  
  
  The rest of the equivalences of (i)-(iv) follow as in the complex unital case in \cite{BComm}.
  For example the equivalence of (ii) and (iii) is  as in \cite{BComm}, and this  without the `idempotent'  or faithful assumption. 
  If (iv) holds
  then so does (iii) and (ii) with $A$ replaced by {\rm Sym}$(A)$, and since the latter is also an idempotent 
  or faithful algebra isomorphic to $A$ 
  we can  apply (iii) or (ii) to it to  imply  (i).  
  
  (iii)  $\Rightarrow$ (iv)\ This works without the `idempotent'  or faithful assumption. 
It follows   by viewing Sym$(A)$ as a subalgebra of $A \oplus B$ via the 
completely isometric homomorphism $a \mapsto (a,a)$, where $B = A$ with its usual product but with $A^{\circ}$ operator space structure.   
By hypothesis $B$ is an operator algebra. 
  
(v) $\Rightarrow$ (i)\ This is as in the lines above  \cite[Corollary 2.5]{BWinv}, and does not require the `idempotent' or faithful condition.   

(iv) $\Rightarrow$ (v)\ Finally, if (iv) holds then so does (i), so that ${\rm Sym}(A)$ is symmetric and commutative. 
  In the complex case the proof of \cite[Corollary 2.5]{BWinv} 
  shows that ${\rm Sym}(A)$ is completely 
  isometrically isomorphic to an algebra of 
 operators whose matrix (representation) with respect to some fixed orthonormal basis is 
symmetric.  This gives (v).  In the real case Proposition \ref{ciop} gives 
${\rm Sym}(A_c) = ({\rm Sym}(A))_c$.  So (iv) implies ${\rm Sym}(A_c)$ 
is (completely isometrically) an operator algebra, and hence by the complex case 
${\rm Sym}(A_c)$, and therefore also its real subalgebra ${\rm Sym}(A)$, is real isometrically isomorphic to an algebra of 
 operators on a complex Hilbert space whose matrix (representation) with respect to some fixed orthonormal basis is 
symmetric. 
\end{proof} 

{\bf Remark.}  In the general case (i) implies (ii), and (ii) is equivalent to (iii), just as in  \cite{BComm}.   As we already said, (v) implies (i), and  (ii) or (iii) implies (iv). 
The latter condition is equivalent to {\rm Sym}$(A)$ is (completely isometrically) an operator algebra with reversed  multiplication of $A$. 

\bigskip

Examining the proof one sees one may relax the 
`faithful' condition in the theorem to simply requiring that  the commutator ideal 
  $J = [A,A]$ acts faithfully on $A$ on the left or on the right.   If this holds we shall say that $A$ is {\em c-faithful}. 
  For if this holds then since $A$ is 3-commutative we have e.g.\ $JA = (0)$, so $J = (0)$, which forces $A$ to be commutative.

    \begin{corollary} \label{comm}  Symmetric real or complex operator algebras are reversible.  
    Conversely if $A$ is a reversible 
    operator algebra, then Sym$(A)$ is a symmetric  operator algebra.    
    Idempotent or c-faithful real  (resp.\ complex)   operator algebras  which are 
    symmetric are commutative algebras.   \end{corollary}
     
     \begin{proof}  The first assertion follows from Proposition \ref{p1}. 
     The rest is obvious from the proof of the last result, and from the remarks after it. 
 \end{proof} 

  \begin{corollary} \label{chsym}  A  real  (resp.\ complex)
 operator algebra $A$  is commutative and symmetric if and only if $A$ is completely isometrically isomorphic 
 to a real  (resp.\ complex) algebra of 
 operators on a complex Hilbert space whose matrix (representation) with respect to some fixed orthonormal basis is 
symmetric  (that is the matrices are symmetric in the undergraduate sense). 
This is also equivalent to:
there exists a conjugation $c$ on a complex Hilbert space $H$ on which $A$ may be completely isometrically represented 
as a real  (resp.\ complex) operator algebra,  such that $c a c = a^*$ for all $a \in A$ (here we are identifying $A$ with its image in $B(H)$).
That is, if and only if  there exists a  real  (resp.\ complex) 
completely isometric algebra representation of $A$ on a complex Hilbert space $H$ as $c$-symmetric operators
on $H$ for a  conjugation $c$.    \end{corollary} 

  \begin{proof}  This deep result is
  from \cite{BWinv}; see the remarks above Theorem \ref{thmsym}.   The condition after the `if and only if' is equivalent to the orthonormal basis/symmetric matrix
  condition in Corollary \ref{comm}, as in the proof of \cite[Corollary 2.5]{BWinv}, or our Theorem \ref{thmsym}.   \end{proof} 

Similarly, the following two results follow by applying the complex case 
of  Corollary \ref{chsym}  as in 
\cite{BWinv} to $A_c$.
  
 \begin{corollary} \label{chsym2}  A  real  (resp.\ complex)  operator algebra $A$  is commutative
if and only if it is isometrically isomorphic to a real  (resp.\ complex)  algebra of  matrices 
(possibly of infinite size, with 
bounded operator norm, and possibly complex 
entries)  that equal their transpose.
\end{corollary}
 
 \begin{corollary} \label{chsym3}  The algebra generated by any operator on a real Hilbert space is
 isometrically isomorphic to the real algebra generated by a complex symmetric operator 
on a complex Hilbert space $K$. That is,  isometrically isomorphic to the real operator algebra generated by an operator on $K$ whose matrix (representation) with respect to some fixed orthonormal basis is a 
symmetric matrix. 
\end{corollary}

{\bf Remarks.} 1)\ In the real case of the results above involving complex symmetric matrices, it is not always possible to use real symmetric matrices.
An example of this is $\bC$ viewed as a real $C^*$-algebra. 

\smallskip

2)\ The real case of several other results in \cite{BWinv} should follow similarly by
using some of the ideas above.

\begin{proposition} \label{wiff} A  symmetric  (resp.\ reversible) operator algebra $A$ is commutative  if and only if $A^1$ is symmetric  (resp.\ reversible).  \end{proposition}  

  \begin{proof}  Since unital reversible operator algebras are commutative the one direction is immediate. 
  Conversely if $A$ is commutative  and symmetric then $A^1$ is  symmetric  by e.g.\ Corollary \ref{chsym}. \end{proof} 

The last result is a necessary and sufficient condition for commutativity of $A$.  We recall that
the unitization $A^1$ is unique up to 
        completely isometric isomorphism.   So in some senses this is the end of the story.
However 
in a given example one may not have a good handle on the operator space structure of $A^1$, indeed the matrix norms 
on the added dimension are usually subtle.  Hence we continue to seek conditions for commutativity only involving the 
operator space structure of $A$. 

\section{A `standard position', and three-commuting algebras}  \label{sp}

\begin{lemma} \label{both} Suppose that $A$ is an operator algebra and that $I(A)$ is represented as a subTRO of $B(H)$
(we are not requiring $A$ to be a subalgebra of $B(H)$).  Then 
we may represent $A$ completely isometrically and as a subalgebra of $B(H^{(2)})$ with a copy  of 
$I(A)$ in $B(H^{(2)})$ represented as a subTRO  and  subalgebra of $B(H^{(2)})$. \end{lemma}

\begin{proof}  Let $Z = I(A)$, a subTRO of $B(H)$.  By Theorem \ref{KP} there exists unique $z \in {\rm Ball}(Z)$ with
the product of $A$ being $x z^* y$. 
Define $\Phi : Z \to B(H^{(2)}) = M_2(B(H))$ to have first row $x z^*$ and $x(1-z^*z)^{\frac{1}{2}}$, and other row $0$.
Note that $\Phi(x) \Phi(y)^*$ has 1-1 entry $xy^*$ for $x,y \in Z$, and other entries 0.  Then $\Phi$ is a faithful, so  completely isometric, ternary morphism.  So the copy $\Phi(I(A))$  in $B(H^{(2)})$ is a subTRO. 
It is easy to check that $\Phi(x) \Phi(y) = \Phi(xz^* y)$ and also $\Phi(x) \Phi(y) =  \Phi(x) \Phi(z)^* \Phi(y)$, for $x,y \in Z$.
\end{proof}

{\bf Remark.}  If $K = H^{(2)}$   then in the copy of $Z$ in  $B(K)$ we have $xy = x z^* y, x, y \in Z$.   
Here $z$ is the image under $\Phi$ of the $z$ in the last proof.  Thus $B(K)$ has two  operator algebra 
products, $xy$ and  $x z^* y$, and these  coincide on $Z$ and on $A$.

\bigskip

If $A$ is  represented completely isometrically and as a subalgebra of $B(K)$, with a copy of 
$I(A)$ in $B(K)$ represented as a subTRO and subalgebra $Z$ of $B(K)$, then we will say that $A$ is in {\em standard position}.

The following associates canonical commutative operator algebras with any reversible operator algebra $A$, and shows that $A$ itself is at least 3-commutative. 

\begin{theorem} \label{both2} Suppose that $A$ is a reversible operator algebra.   
 \begin{itemize} 
\item[(1)]  We may represent 
$A$ completely isometrically and as a subalgebra of $B(K)$, with a copy of 
$I(A)$ in $B(K)$ represented as a subTRO and subalgebra $Z$ of $B(K)$, such that there exists unique $w, z \in {\rm Ball}(Z)$ with
$yx = y z^*x = x w^* y$ for $x, y \in A$ (all operations here the natural ones in $B(K)$).   
\item[(2)]    $Az^*, Aw^*, z^* A, w^* A$ are commutative operator algebras.
\item[(3)]  Viewed completely isometrically as commutative subalgebras $$L(Az^*), L(Aw^*), R(z^* A), R(w^* A)$$
of $CB(A)$, all of these algebras mutually commute, and we have $L(Az^*) = R(w^* A)$ and $L(Aw^*) = R(z^* A).$
\item[(4)]  $A$ is 3-commutative:\ any  product of three or more elements of $A$ commutes.   
\item[(5)]  If $n \geq 3$ and $a_1, \cdots , a_n \in A$, then in any  product 
$a_1 z^* a_2 z^* \cdots z^* a_n$,  any selection of the $z$'s may be changed to $w$'s without changing the 
value of the product.
\item[(6)]  $A$ is commutative if and only if $z = w$ in {\rm (1)}. 
 \end{itemize} 
\end{theorem}

\begin{proof}  (1)\ Let $n(x,y) = yx$.  Then $(A,n)$ is completely isometrically isomorphic to an operator algebra. By Lemma \ref{both}
we may represent 
$A$ completely isometrically and as a subalgebra of $B(K)$, with a copy of 
$I(A)$ in $B(K)$ represented as a subTRO $Z$ of $B(K)$, such that there exists unique $z \in {\rm Ball}(Z)$ with
$xy = x z^*y$ (all operations here the natural ones in $B(K)$).   We also have $y z^* x = x w^* y$ in $Z$ 
for a unique $w \in {\rm Ball}(Z)$ and all $x, y \in A$, by Theorem \ref{KP} 
applied to the product $n(x,y) = yx = y z^* x$ on $A$.   

(2)\ Note that $x z^* yz^* a = x z^* (a w^* y) = (x z^* a) w^* y = y z^* x z^* a$ for $a,x,y \in A$.   By \cite[Proposition 4.4.12]{BLM}  we have 
 $x z^* yz^* = y z^* x z^*$.   So $Az^*$ is a  commutative operator algebra.  Similarly $Aw^*, z^* A, w^* A$ are
 commutative.   
 
 (3)\ Viewing $Az^*$ in the operator space left multiplier algebra $\cM_\ell(A) \subseteq CB(A)$ (see \cite[Chapter 4]{BLM}) completely isometrically, and $w^* A$ in the left multiplier algebra
$\cM_r(A)$ it is clear that $L(Az^*) = R(w^* A)$, and they commute.   So $Az^*$ and $w^*A$ are completely isometrically isomorphic.
Similarly for $Aw^*$ and $R(z^* A)$.  
 
(4)\ It follows that in a product $a_1 z^* a_2 z^* \cdots z^* a_n$ we can 
 move the terms $a_k z^*$ around.   Also, this product equals $$a_2 z^* \cdots z^* a_n w^* a_1 = a_2 z^* \cdots z^* a_1 z^* a_n
 = a_2 z^* \cdots z^* a_n z^* a_1.$$ Thus any  product of three or more elements of $A$ commutes.

(5)\ This  follows from the above, since for example any $z^* a_k z^* a_j = z^* a_j z^* a_k = z^* a_k w^* a_k$. 

(6)\ One direction of this is obvious.
If $A$ is commutative 
then (1) implies that $x z^*y = x w^* y$ for $x, y \in A$.  By the uniqueness in (1), this implies $z=w$.
\end{proof} 

{\bf Remarks.} 1) For an operator algebra  $A$, being  left faithful is equivalent to $x \mapsto xz^*$ being one-to-one (note that this map is a homomorphism from $A$ into the operator algebra $Az^*$).  Indeed $aA = 0$ if and only if $xz^* A = 0$, and so if and only if $xz^* = 0$.  Similar statements apply to `right faithful'.

\medskip

2)\ Not every  3-commutative operator algebra is reversible. E.g.\ the strictly upper triangular $3 \times 3$ matrices is 3-commutative.
However it can be shown to be not reversible (e.g.\ using Theorem \ref{both2}
and explicitly showing that there can exist no 
$w$ there, as explained below Theorem \ref{BRSKP}).   Indeed one can show that all reversible subalgebras of $M_3$ are commutative.  Questions like these could make interesting undergraduate research projects.   

\medskip

As said earlier, there are interesting links to triangular operator algebras in the sense of Radjavi and Rosenthal \cite{RR} and others. 

\begin{theorem} \label{Burn} Every 3-commuting complex matrix algebra $A$ is  triangularizable.   That is, $A$  is unitarily equivalent to an algebra of upper triangular matrices. \end{theorem}  

\begin{proof}   We first prove 
that  $A \subseteq M_n$  has a common  eigenvector.
By Burnside's theorem \cite{Lam}
there is a nontrivial invariant subspace $K$.    
If $A_{|K}$ is $(0)$ then $A$ has a  common eigenvector in $K$ (any nonzero vector in $K$).
  If $A_{|K}$ is a full matrix algebra $M_m$ then the product of three elements cannot be commutative.
  So $A_{|K}$ is contained properly in a full matrix algebra.   We can repeat this procedure until we obtain a nontrivial invariant subspace
$K$ with $A_{|K} = (0)$ or  $K$ is one dimensional.  In both cases we have a common  eigenvector for $A$. 

Let $\zeta_1$ be a normalized common eigenvector, and $K = H \ominus \{ \zeta_1 \}$. Let $B = B(H)$.  By Sarason's lemma (see e.g.\ \cite[Proposition 3.2.2]{BLM}, in the real case the same easy proof appears to work) 
$R = P_K A _{|K}$ is a 3-commuting subalgebra of $B(K) \cong p B(H) p$.   
Indeed $P_K a P_K b P_K c_{|K}  = P_K abc_{|K}$.   Since $A$ is 3-commuting we can interchange $a,b$ and $c$ everywhere in the last equation without change.   By the argument above 
$R$ has  a normalized common eigenvector.  
Continuing in this way we see that $A$ is  upper triangular.
\end{proof}

\section{Counterexamples} \label{kegs}  

\begin{proposition} \label{acre} An anticommuting operator algebra is reversible.  If $B$ is a commutative operator algebra and 
$C$ is an anticommuting operator algebra, then any subalgebra of $B \oplus^\infty C$  is reversible. 
\end{proposition}  
\begin{proof}   An anticommuting operator algebra is reversible by Theorem \ref{BRSKP}, taking $z = -I$ there.   The rest follows from Proposition \ref{classof}. 
\end{proof}

Our counterexamples all derive from one basic example of an anticommuting, hence reversible, noncommutative  operator algebra. 
But first we recall the canonical anticommutation relations (CAR).  Consider the closed linear span of a family of operators $(a_i)_{i \in I}$ in $B(H)$ satisfying the CAR:
$$a_i a_j + a_j a_i = 0 \; ,  \; \; \; \; a_i a_j^* + a_j^* a_i = \delta_{ij} I.$$
This is a Hilbert space, and an operator space, which  up to a complete isometry depends only on the cardinality of $I$.

\begin{example} \label{Ex1} Let $\Phi_2$  be the 2 dimensional Hilbertian operator space modeled on the 
canonical anticommutation relations (CAR).  We represent it as upper triangular 
matrices  in $M_4$ via the pair of partial isometries
$$U = e_{13} + e_{24}, \; \; \;  V = e_{12} - e_{34}.$$ 
   These satisfy the usual form of the canonical anticommutation relations.   So for example 
$UV = - VU, U^2 = V^2 = 0$, $U^* V = -V U^*$.   Let $W = VU = e_{14}$. 
Let $A = {\rm Span} \{ U,V,UV \}$, the subalgebra generated by $U, V$ in $M_4$.    This is an operator algebra.   In some sense we do not care about the particular representation, since it is well known that even the $C^*$-algebra generated by operators satisfying CAR is unique, hence the generated closed algebra will be unique up to completely isometric isomorphism.    However this particular representation is very convenient, and is upper triangular which makes several things more transparent. 

We have 
$$(\alpha U + \beta V + \gamma UV) \, (\alpha' U + \beta' V + \gamma' UV) = 
(\alpha \beta' - \beta \alpha') \, UV = (\alpha \beta' - \beta \alpha') \, W, $$
for scalars $\alpha, \alpha', \beta,  \beta,  \gamma, \gamma'$.   Thus $xy = -yx$ for $x, y \in A$.  That is, $A$ is an anticommuting operator algebra.  By e.g.\ Proposition \ref{acre}, or by Theorem \ref{BRSKP}, we see that  $A$ is reversible.   In fact  $A$ is a symmetric operator algebra: indeed 
there is a $*$-automorphism $\theta$ of $M_4$ such that $\theta(x) = -x^T$ for all $x \in A$.  (Namely conjugation 
$\theta(x) = w^* xw$ by the unitary $w = e_{41} + e_{32}
- e_{23} - e_{14}$.)   Thus $$\| [x_{ji} ] \| = \| [ \theta(x_{ij}) ] \| = \| [ x_{ij} ] \| , \qquad x_{ij} \in A .$$  However  of course $A$ is not commutative: $UV \neq VU$.

We compute the injective envelope $I(A)$, in part because we shall need the details below elsewhere.   Let $Z$ be the TRO generated by $A$ in $M_4$, the span of products of the form $a_1 a_2^* a_3 a_4^* \cdots a_n$ for $a_k \in A.$ 
Consider the projections $p_1 = UU^*, p_2 = VV^*, q_1 = p_1^\perp = U^*U, q_2 = p_2^\perp = V^*V$, and set
$p = p_1 \vee p_2   = I_3 \oplus 0 \in Z Z^*$ and $q = q_1 \vee q_2  = 0 \oplus I_3 \in Z^*Z$.   
Note that $WW^*, V UU^*$ and $VUV^*$ are scalar multiples of $e_{11}, e_{12}$ and $e_{13}$ in $A A^*$.   From this 
it is an easy exercise that $C^*(A A^*) = p M_4 p$ and $Z = p M_4 q$. 
 Thus $p M_4 q \cong M_3$ completely isometrically
(ternary) isomorphically.   Indeed $p$ and $q$ commute, so $pq$ is a projection, indeed $pq = 0 \oplus I_{2} \oplus 0 \in M_4$. 
 
 Since $p M_4 q$ is a `simple TRO' (or equivalently, $p M_4 p$ is simple), it follows that the ternary envelope of $A$ is $Z = p M_4 q$. 
 This is also the injective envelope since it is injective.    That is, $I(A) = p M_4 q$. 
 Note that $I(A)  \cong M_3$ completely isometrically (ternary) isomorphically.  Indeed 
$$Z = I(A) = \begin{bmatrix} 0 & * & * & *  \\ 0 &   * & * & * \\ 0 &   * & * & * \\ 0 &   0 & 0 & 0  \end{bmatrix}
 \subset M_4.$$
(It is interesting that  $A$ cannot be represented completely isometrically as a subalgebra of $M_3$, since that would imply that $A$ is commutative by Remark 2 after Theorem \ref{both2}.  We also remark that the $C^*$-algebra generated by $A$ can be shown to be  all of $M_4$, but this will not be used.) 

We can also identify the elements $z,w$ from Theorem \ref{both2} in this case.  Indeed $P(x) = pxq$ is the canonical projection onto $Z = I(A)$,
and so $z = P(1) = pq$.  And $w = -z = -pq$ in this example.  So $z \neq w$. 

Finally we remark that all the matrices considered in this example have real entries, and so the example has the same conclusions 
in the real case. 
We conjecture that up to unitary equivalence $A$ is the only noncommutative reversible  $4 \times 4$ matrix algebra.  Again this may be a nice undergraduate research project. 
\end{example}

\begin{corollary}  \label{uni}  If $A$ is a symmetric operator algebra  then its unitization $A^1$ need not be symmetric.
If  $A$ is a reversible operator algebra  then $A^1$ need not be a reversible operator algebra.
 \end{corollary}

\begin{proof}  This follows from the example and 
Proposition \ref{wiff}. 
 \end{proof}

To increase the dimension of our counterexample a first thought might be to 
generalize our basic example to $\Phi_n$, the $n$th Hilbertian operator space modeled on the canonical anticommutation relations (CAR).  Note that  $\Phi_n$ is symmetric. In other words if one considers $A$, the algebra generated by $\Phi_n$ inside the CAR $C^*$-algebra
$C^*(\Phi_n)$, represented in $M_{N}$ say, 
then is $A$ with 
reversed multiplication an operator algebra?   Is $A$ a symmetric  operator algebra?   
Is there an automorphism $\theta$ of $C^*(\Phi_n)$ such that  $$\| [x_{ji} ] \| = \| [ \theta(x_{ij}) ] \| = \| [ x_{ij} ] \| , \qquad x_{ij} \in A .$$
It can be  checked by a direct computation (which we omit) that the answers to all these questions are in the negative.   

One may however fix some of the issues in the last paragraph, while increasing the dimension,  by considering a 
direct sum of several copies of the algebra in Example \ref{Ex1}.  However it is more interesting to consider other known examples of families of anticommuting matrices (see e.g.\ \cite{H}):

\begin{example} In $M_{2n+2}$ consider matrices  $u_1, \cdots u_{2n}$ defined by
$$u_{2i-1} = \begin{bmatrix} 0 & \vec w_i^\tran & 0 \\
 0 &   0 & \vec v_i  \\ 0 &    0 &   0  \end{bmatrix} , \; \; 
 u_{2i} = \begin{bmatrix} 0 & \vec c_i^\tran & 0 \\
 0 &   0 & \vec d_i  \\ 0 &    0 &   0  \end{bmatrix} , \qquad i = 1, \cdots , n ,$$ 
 where $$\vec w_i = \vec e_i \otimes \vec e_1 , \; \; \vec v_i = \vec e_i \otimes \vec e_2 , \; \;  \vec c_i = \vec e_i \otimes \vec e_2 , \; \; 
 \vec d_i = - \vec e_i \otimes \vec e_1 .$$  The latter are all column vectors in $\bC^n \otimes \bC^2$.  The $(u_i)$ anticommute, in particular 
 $u_{2i-1} u_{2 i} = - u_{2 i} u_{2i-1}$ for all $i = 1, \cdots , n$, and this is a scalar multiple of $w = e_{1,(2n+2)}$, while  $u_i u_j = 0$ if 
 $|i-j| > 1$.  
 The algebra $A$ generated by the $(u_i)$ is $$A = {\rm Span} \{ w , u_1, \cdots u_{2n} \}.$$ 
 As in Example \ref{Ex1} one checks that $xy = -yx$ for $x, y \in A$.  That is, $A$ is an anticommuting operator algebra, so again
  $A$ is reversible but $A$ is not commutative.

We compute the injective envelope $I(A)$.   Let $Z$ be the TRO generated by $A$ in $M_{2n+2}$, the span of 
products of the form $a_1 a_2^* a_3 a_4^* \cdots a_n$ for $a_k \in A.$ 
Consider the projections $p_i = u_i u_i^*$ and $w w^* = e_{11}$ we see that $e_{ii} \in Z Z^*$ for all $i = 1, \cdots, 2n+1$. 
So $p = \vee_i \, p_i   = I_{2n+1} \oplus 0 \in Z Z^*$.     Similarly on the right as in Example \ref{Ex1} we obtain
$q = 0 \oplus I_{2n+1}$. 
Note that $ww^*, w u_i^*$ are scalar multiples of $e_{11}, e_{12}, \cdots , e_{1, (2n+1)}$ in $A A^*$.   From this 
it is an easy exercise as in Example \ref{Ex1}  that $C^*(A A^*) = p M_{2n+2} p \cong M_{2n+1}$ and $Z = p M_{2n+2} q \subset M_{2n+2}$. 
 Thus $Z = C^*(A A^*)  A \cong M_{2n+1}$ completely isometrically
(ternary) isomorphically.   Again $p$ and $q$ commute, so $pq$ is a projection, indeed $pq = 0 \oplus I_{2n} \oplus 0 \in M_{2n+2}$. 
 
 Again by simplicity as in Example \ref{Ex1}, it follows that the ternary and injective envelope of $A$ is $Z$, and 
 we can identify the elements $z,w$ from Theorem \ref{both2} in this case again as
 $z  = pq$ and $w = -z = -pq$.    \end{example}

  {\bf Remark.} Playing with such examples yields many interesting questions, some of them very suitable for undergraduate research projects.   For instance, proving that $A$ in the last example  is a symmetric operator algebra.  Or one could make many variants of the example.  For example in $M_{n+2}$ 
 consider matrices $u_1, \cdots u_{n}$ defined by
$$u_{i} = \begin{bmatrix} 0 & \vec e_i^\tran & 0 \\
 0 &   0 & \vec v_i  \\ 0 &    0 &   0  \end{bmatrix} , \; \; \; \;  \vec v_k = (-1)^{k+1} \, \vec e_{k+1},   \qquad i = 1, \cdots , n .$$ 
 The $(u_i)$ no longer anticommute.  The algebra $A$ generated by the $(u_i)$ is $A = {\rm Span} \{ e_{1,n+2} , u_1, \cdots u_{n} \}$.
 One may ask if such algebras are commutative, reversible or symmetric, etc.

\begin{example} One might conjecture in view of Proposition \ref{acre} that every reversible (resp.\ symmetric) 
operator algebra is a subalgebra of $B \oplus^\infty C$, where $B$ is a commutative and 
$C$ is an anticommuting operator algebra.   However it turns out that one can modify Example \ref{Ex1} in a nontrivial way as 
below  to obtain  a counterexample to this.

We say that $x \in A$ is a {\em strictly anticommuting element} if $xy = -yx$ for all $y \in A$, and if $xy \neq 0$ for
some $y \in A$.  One may ask if a reversible operator algebra is commutative if and only if it has no strictly 
anticommuting elements.   Again, the following example contradicts this. 

Let $N = M_m$ or $B(H)$.  Consider an isometry  $s \in N$ and the two  matrices $u = e_{12} \otimes I + e_{34} \otimes I$ and $v = e_{13} \otimes I +  e_{24} \otimes s$ in $M_4(N)$.  
We have $u^2 = v^2 = 0$. 
Let $A$ be the span of $\{ u,v,uv,vu \}$  in $M_{4}(N)$.  This is a 3-nilpotent algebra, and  
$$vu = e_{14} \otimes I , \; \; \;  uv = e_{14} \otimes s .$$ 
Consider $w = 0 \oplus r \oplus s \oplus 0$ for a contraction $r$.  Then $uwv = u (e_{24} \otimes rs) = vu$ if $rs = 1$, so that $s$ is 
an isometry (unitary if $B = M_m$), and $r = s^*$. 
Also $v w u = v (e_{34} \otimes s) = uv$. For $x, y \in A$ one can check that  $$yx = y_1 x_2 uv + y_2 x_1 vu,$$
where a $1$ (resp.\ 2) subscript denotes the coefficient of $u$ (resp.\ $v$).
Also $x w y = x_1 y_2 uwv + x_2 y_1 v w u$. 
So $yx  = x w y$.  It follows from Theorem \ref{BRSKP} that $A^{\rm rev}$ is an operator algebra.  That is, $A$ is reversible.   
It is easily checked that $A$ has no strictly anticommuting elements.  We do note however that this 
$A$ always has pairs of anticommuting elements: elements $x,y$ with $xy=-yx \neq 0$.  Indeed this holds for 
$x = u + v, y = u-v$.    It would be interesting to find an example with no such anticommuting  pair. 

In $A$ we have $\| \alpha uv + \beta vu \| = \| \beta I + \alpha s \|$.    If $s = iI$ for example, and $\alpha, \beta$ real, then this 
is $\sqrt{\alpha^2 + \beta^2}$. 
On the other hand if $A$ were a subalgebra of $B \oplus C$  as stated
then we may write $u = (u_1, u_2), v = (v_1, v_2)$.
We have $uv = (u_1 v_1 , u_2 v_2 ),$ and $vu = (u_1 v_1 , -u_2 v_2 ).$ So
$$\| \alpha uv + \beta vu \| = \| ((\alpha  +\beta)  u_1 v_1 , (\alpha  - \beta)  u_2 v_2) \| = \max \{ |\alpha  +\beta | 
\| u_1 v_1 \| , |\alpha  -\beta |  \| u_2 v_2 \| \}.$$  If $\| u_1 v_1 \| \geq  \| u_2 v_2 \|$
and $\alpha > \beta > 0$, then this is $(\alpha  +\beta)  \| u_1 v_1 \|$.  If  $\| u_1 v_1 \| <  \| u_2 v_2 \|$
and $\alpha > 0 > \beta$, then this is $|\alpha  -\beta |  \| u_2 v_2 \|$.  In either case we have a contradiction.    We leave it to the reader to check if such examples are always symmetric. 
\end{example}

  \section{Essential extensions of operator algebras} \label{sess}

  It should be interesting to develop a theory of 
 operator algebras whose injective envelope is a $C^*$-algebra $D$.   
 If  $A$ is of this type, and if $D$ can be chosen so that $A$ is a subalgebra of  $D$, then we shall say that $A$ is an {\em  essential operator algebra}.  
 This includes all unital and approximately unital operator algebras.  The reason for this notation is the following characterization. We recall that if $E \subseteq F$ are operator spaces then $F$ is an {\em essential  
 extension} of $E$ if a complete contraction 
 $u : F \to B(H)$ is completely isometric if and only if $u_{|E}$ is completely isometric.  (This should hold for all $H$ and such $u$.)  We say  that $F$ is a  {\em rigid  
 extension} of $E$ if a complete contraction 
 $u : F \to F$ is the identity map if and only if $u_{|E} = I_E$.
 
\begin{lemma} \label{ess}   Let $A$ be  an operator  algebra. Then $A$ is essential
   if and only if  the unitization of $A$  is an essential operator space extension of $A$.  
   In this case $A^1$ is also a rigid operator space extension of $A$.
   \end{lemma}
  
\begin{proof}   It is well known (and an exercise) that $F$ is an 
essential operator space extension of $E$ if and only if $F \subseteq I(E)$ (that is, there exists 
a linear complete isometry $\theta : F \to I(E)$ extending the inclusion $E \to I(E)$), and if and only if $I(F)$ is
an injective envelope of $E$.   
Thus $A^1$  is an essential operator space extension of $A$ if and only if $I(A^1) = I(A)$.  Now $I(A^1)$ may be chosen to be a  $C^*$-algebra containing $A^1$, and hence $A$, as a subalgebra.  So $A$ is essential.
Conversely if $A$ is essential, a subalgebra of $D$ as above,
then $A^1 = A + \Fdb 1_D$ by uniqueness of the unitization.  So $A^1$ is an essential operator space extension of $A$. 

It is an ancient fact (and again a pleasant exercise) that every essential extension is rigid (this is seemingly due to Isbell in the 60's in the Banach space category, and was observed by Hamana for operator spaces).   
\end{proof}

\begin{corollary} \label{essc} An essential operator algebra is commutative if and only if it is reversible.  \end{corollary}

\begin{proof}   The product on $A$ is
$ab = az^* b = a 1_D^* b$ for $a, b \in A$ in the language used before (e.g.\ in the proof of Theorem \ref{both2}). So  $z = 1_D$, by the same considerations 
used there.  Suppose that $a A = a z^* A = 0$.
Then by \cite[Proposition 4.4.12]{BLM} we have  $a z^* = 0 = a$. 
So $A$ is a faithful algebra, hence the result follows from Theorem \ref{mainc}.  \end{proof}

In the light of the last fact in the Lemma, it is natural to ask if it is rigidity rather than essentiality that is causing the commutativity in Corollary \ref{essc}.
That is, is an  operator algebra $A$ which is {\em rigid} in its unitization, commutative if and only if it is reversible?   Or even perhaps is $A$ essential (that is, $A^1$ is an essential extension of $A$) if and only if $A^1$ is a rigid extension of $A$?  
Or perhaps this is true if $A$ is finite dimensional. 
There are almost no examples in the literature, in any category, of 
rigid extensions which are not essential extensions (we said at the end of the proof of Lemma \ref{ess} that all essential extensions are automatically rigid).  
In fact the only example in the literature which we are aware of
at the present time is in a category not relevant for us, and even there it is not said explicitly. 
   The following shows that the answer to all these questions is in the negative. 

\begin{theorem}  There exists a three dimensional 
symmetric  
noncommutative matrix algebra $A$ in $M_4$ which is rigid 
in its unitization $A^1$, but for which $A^1$
is not an essential extension of $A$. 
   \end{theorem} 

  \begin{proof}  
  As the reader may suspect, $A$ is our main Example \ref{Ex1} in Section \ref{kegs}, a subalgebra of $M_4$.  Claim: $A^1$ is a rigid but not essential operator space extension of $A$.  By Corollary \ref{essc}  $A^1$ cannot be an essential extension since $A$ is noncommutative.
  Another way to see this is that $I(A^1) = M_4$ and $B = I(A)$ have different dimensions, in contrast to a fact in the proof of Lemma \ref{ess}. We saw in Section \ref{kegs} that $B$ is the copy of $M_3$ in the top right of $M_4$.
  
To see that $A^1$ is a rigid operator space extension of $A$ let $u : A^1 \to A^1$ be completely contractive with $u_{|A} = I_A$.
By Hamana's theory in this finite dimensional case, by taking limits of powers of $u$ we may assume that $u$ is a completely contractive projection.   Thus it clearly suffices to show that  $u(A^1) \neq A$.
By way of contradiction suppose that $u$ is a projection onto $A$. 
Extend $u$ to $P : M_4 \to B$.   Then $P  \circ i_B = I_B$ since $B = I(A)$ and $(P  \circ i_B)_{|A} = I_A$.
Thus $P$ is a completely contractive projection.  By Youngson's theorem \cite[Theorem 4.4.9]{BLM}, and  by  the fact that $B$ is a subTRO of $M_4$, we have that 
$$a P(1)^* b = P(a P(1)^* b) = ab,  \; \; \; P(1) a^* b = P(P(1) a^* b) = P(a^*b),$$ and so on.    Thus by 
\cite[Proposition 4.4.12]{BLM}   we have
$a P(1)^* = a z^*$, and in particular $U P(1)^* = U z^*$, where $z = 0 \oplus I_2 \oplus 0$
as in Example \ref{Ex1}.  However this is impossible since
$P(1) \in {\rm Span} \{ U,V,UV \}$, and $U U^* = e_{11} + e_{22}, UV^* = -e_{23}, UV^* U^* = - e_{23}(e_{31} + e_{42}) = - e_{21}$ , while 
$U = e_{13} +  e_{24}$. 
\end{proof}

\section{Jordan operator algebras} \label{joa}   We first use 
Example \ref{Ex1} to answer in the negative a question raised in \cite{BNjmn}, above Corollary 2.2 there, namely if the correspondence alluded to there between Jordan operator algebra products on $X$ and certain elements $z \in {\rm Ball}(I(X))$,  is a bijection.   One way to phrase this is: for an operator space $X$, if $z, w \in {\rm Ball}(I(X))$ satisfy 
$a z^* a = a w^* a$ for all $a \in X$, then is $z = w$?     This fails 
for example in Example \ref{Ex1}, if $z, w$ are as in that example, indeed $a^2 = a z^* a = a w^* a = 0$ here.

  \begin{theorem}  If $A$ is a reversible operator algebra with a unique Jordan algebra 
  unitization up to complete isometry, then  $A$ is commutative.      \end{theorem} 

  \begin{proof}  
 If $A$ is a reversible operator algebra with a unique Jordan unitization up to complete isometry, then 
 $A^1 \cong (A^{\rm rev})^1$ unitally completely isometrically and as Jordan algebras, via a unital map $\theta$ 
 which is the identity on $A$.  Hence by the 
 Banach-Stone theorem for unital operator algebras (e.g.\ \cite[Corollary 8.3.13]{BLM}) $\theta$ is a homomorphism.  So 
 $A$ is commutative. 
 \end{proof} 
 
 {\bf Remark.} 1)\ 
 Indeed the operator algebra $F_2$ in \cite[Proposition 2.1]{BWjr} is a 
commutative operator algebra with two Jordan unitizations $F_2 + \Cdb I_6$ and $\Cdb I_4 + E_2$.   The first of these is also an operator algebra unitization.  
It is shown there that these are not Jordan  completely isometrically isomorphic.  Note that the Jordan algebra
$\Cdb I_4 + E_2$ is not 
an operator algebra. These Jordan  algebras are not unitally completely isometric, or else the canonical 
extension $C^*_{\rm e}(F_2 + \Cdb I_6) \to C^*_{\rm e}(E_2 + \Cdb I_4)$ would be a $*$-homomorphism.
This would imply that $\Cdb I_4 + E_2$ is a subalgebra of $C^*_{\rm e}(E_2 + \Cdb I_4)$.

\medskip

 2)\ We may also improve the counterexample given in \cite[Section 2]{BWjr}, a counterexample  to one of the major questions posed at the end of 
 \cite{BWjmn} (about the uniqueness of the  Jordan operator algebra unitization).    Indeed if $A$ is as in Example \ref{Ex1}, for example, and $B = A^{\rm rev}$, then $A=B$ 
completely isometric as Jordan  algebras.  However if the canonical unital extension $A^1 \to B^1$ were 
a completely isometry then $A$ would be commutative as in e.g.\ the last proof.  This is an improvement in that the Jordan product is nontrivial, and nontrivial in a nonartificial way.

\section{Other sufficient conditions}  \label{last}

\subsection{Wedderburn theory of finite dimensional algebras}

  It is instructive that among finite dimensional algebras the semisimple ones  are precisely the $C^*$-algebras, by a well known theorem of Wedderburn.

\begin{corollary} \label{chss}  A       reversible  semisimple operator  algebra is commutative.  For a general  reversible  (resp.\ symmetric) operator  algebra $A$, the Jacobson radical $J(A)$ is reversible  (resp.\ symmetric) and 
     $A/J(A)$ is commutative and semisimple.  
 \end{corollary}
 
 \begin{proof} In the finite dimensional  case this is obvious from the next result, Theorem \ref{wed}, which also gives more information.
In the general case suppose that $J$ is the commutator ideal of $A$.  This is a nilpotent ideal of $A$ so is contained in 
the Jacobson radical $J(A)$.  If $A$ is semisimple it follows that $J = J(A) = 0$, so that $A$ is commutative.
Otherwise $A/J(A)$ is semisimple by the theory of the Jacobson radical, and so is commutative using Proposition \ref{classof}.  
that result also implies that  $J(A)$ is reversible  (resp.\ symmetric).  \end{proof}

     \begin{theorem} \label{wed}   Let $A$ be  a  finite dimensional 
     reversible  operator  algebra. Then $A$ is the algebra direct sum of two ideals $C$ and $K$ where $C$ is a unital commutative operator  algebra (or $(0)$), and $K$ is a 
  nilpotent ideal and is also reversible.  
  (Of course $CK = KC = (0)$.) If $A$ is a matrix algebra then $K$ may be taken (by choosing an appropriate orthonormal basis) to consist of strictly upper triangular matrices.  \end{theorem}
 
\begin{proof}   If $A$ is not radical (and nilpotent) already, let 
      $B$ be the semisimple subalgebra in the Wedderburn principal theorem
      (see e.g.\ p.\ 163-164 in \cite{Row}), and let $J$ be the maximal nilpotent ideal.
      Then $B$ is unital by Wedderburn's other theorem referred to earlier (since it is a finite dimensional  semisimple algebra), and hence commutative by Theorem  \ref{mainc} above.
     Let $f = 1_B$ then $f = f^2$ is in the center of $A$ by 3-commutativity,
     and $C = Af = fAf$ is a unital subalgebra, so is reversible and hence 
     is commutative by Theorem  \ref{mainc}  above.  
      It is also an ideal.  Now $f \notin J$, by Wedderburn's principal theorem.
     If $a = b + j$ for $j \in J, b \in C$ then $a - fa = j - fj \in J$.
     Thus $K = \{ a - f a : a \in A \}$ is an ideal in $J$, and $K$ is nilpotent since $J$ is, and $A = C \oplus K$.  
     
      If $A$ is a reversible matrix algebra then by Theorem \ref{Burn} 
      $K$ may be taken  to consist of upper triangular matrices.  Since it is nilpotent, these are 
  strictly upper triangular. 
 \end{proof}
 
  This reduces any classification of finite dimensional 
     reversible  operator  algebras to the nilpotent case.

To illustrate how the last theorem may be used, consider the problem of checking that all reversible subalgebras of $M_3$ are commutative.   
Theorem \ref{wed} reduces this to the case of reversible strictly upper triangular subalgebras, and these are easily classified.  They are all one or two dimensional, except for the algebra of all strictly upper triangular matrices which we said earlier is not reversible.

\subsection{The commutator ideal and Lie ideals} When considering  specific examples of reversible operator algebras we have found that often proving some facts about the commutator ideal $[A,A]$ is the key to understanding whether $A$ is commutative.  
In Example \ref{Ex1} we have the commutator ideal consists of scalar multiples of $e_{14}$, for example.  We already used  the commutator ideal  in Section \ref{rsy}, in our use of `c-faithful'. By Theorem 
\ref{both2} (4) we know that $A$ is 3-commutative, so that
$[A,A]$ is contained in the left and right annihilator ideals of $A$. 
It seems that using some Lie ideal theory such as may be found in \cite{Thiel} and the references therein, may be helpful in obtaining other useful characterizations of commutative operator algebras.

If $A$ is 3-commutative then $I = \overline{A^2}$ is an ideal, is commutative, and $Q = A/I$ is commutative, indeed 
    square-zero nilpotent (i.e.\ has trivial product).   Conversely, given a commutative (resp.\ commutative and symmetric) operator algebra,
    and a  operator space  (resp.\  symmetric) $Q$ with trivial product, then $I \oplus Q$ is a 3-commutative (resp.\ commutative and symmetric) operator algebra. 
Proceeding along these lines one may prove for example that if $I= \overline{A^2}$ acts faithfully on $A$  then $A$ is commutative.  This also follows by 
noting that $I$ contains the commutator ideal 
  $[A,A]$, so that $A$ is c-faithful in the sense defined after 
Theorem \ref{mainc}.  Hence $A$ is commutative by the discussion in that place.

\subsection{Ternary matrix algebras}

The following may be viewed as a generalization of Corollary \ref{essc}.

\begin{corollary}  \label{parti}  Let  $A$ be an reversible operator algebra.  If $A$ is a subalgebra of a $C^*$-algebra $B$ and if $A$ generates $B$ as a TRO 
(that is the span of products of the form $a_1 a_2^* a_3 a_4^* \cdots a_n$ are dense in $B$)
then $A$ is commutative.
 \end{corollary}

\begin{proof} Suppose that we have a  $C^*$-algebra $B$ with subalgebra $A$ generating $B$ as a TRO.  Then 
there is an ideal $J$ of $B$ such that the $C^*$-algebra $B/J$ is a ternary envelope of $A/J \cong A$.   Moreover 
$A/J$ is a subalgebra of the $C^*$-algebra  $B/J$, and has the same properties as $A$.   Thus we may replace
$B/J$ by $B$ and $A/J$ by $A$, that is we may assume that $B =  {\mathcal T}(A)$.   Then $I(A) = I({\mathcal T}(A)) = I(B)$ is a unital 
$C^*$-algebra  containing $B$ and $A$ as subalgebras.   Thus  $A$ is commutative by Corollary \ref{essc}.  
 \end{proof}
 
While the last result is interesting, one might ask for a `TRO version', where $B$ or $I(A)$ in the result is a TRO as 
opposed to a $C^*$-algebra.   
 Every  finite dimensional complex TRO is (completely isometrically 
and ternary isomorphic to) 
an `$L^\infty$-direct sum' $\oplus_{k = 1}^r \, M_{n_k,m_k}$.   This suggests the next definition and result.  Note that  even if $n \neq m$ we may 
regard $M_{n,m}$ as an algebra by embedding as the `top left corner' in $M_k$ for $k$ large enough. 
That is, for $x, y \in M_{n,m}$ we have $x \cdot y = [\sum_{k = 1}^t \, x_{ik} y_{kj}] \in M_{n,m}$ where $t = \min \{ n,m \}$.
We define a ternary matrix algebra to be an `$L^\infty$-direct sum' $\oplus_{k = 1}^r \, M_{n_k,m_k}$ of such algebras.   It is 
also of course a TRO as well as being an operator algebra with the product $(x_k) (y_k) = (x_k \cdot y_k)$, where $\cdot$ is as just defined. 

\begin{corollary}  \label{recta}  Suppose that $A$ is a subalgebra of a ternary  matrix algebra $B$ such that $A$ generates $B$ as a TRO
(that is the span of products of the form $a_1 a_2^* a_3 a_4^* \cdots a_n$ are dense in $B$).   
If $A$ is  
reversible  then 
 $A$ is commutative.
 \end{corollary} 

\begin{proof}  Suppose that 
$B = \oplus_{k = 1}^r \, M_{n_k,m_k}$ 
and that $p_k, q_k$ are the 
the left and right support projections of
$M_{n_k,m_k}$ here.  Clearly $A_k = p_k A q_k$  is a  subalgebra of $M_{n_k,m_k}$.  
There is a TRO  ideal $J$ of $B = \oplus_{k = 1}^r \, M_{n_k,m_k}$ such that $B/J$ is a ternary envelope of $A/J \cong A$. 
Indeed this corresponds to removing some of the rectangular blocks $M_{n_k,m_k}$.   We may assume that the removed blocks are `at the end', so that $Z = {\mathcal T}(A) = \oplus_{k = 1}^s \, M_{n_k,m_k}$ for $s \leq r$.  Note that $A$ is still a subalgebra of the latter.  
Let $p_k \in Z Z^*$ be the projection corresponding to $I_{n_k}$  (so the left support projection of the summand $M_{n_k,m_k}$).
Similarly let $q_k$ be the projection corresponding to $I_{m_k}$.
 The TRO  generated by $p_k A q_k$ 
 inside $B$ equals 
 $M_{n_k,m_k}$.  Indeed  certainly this generated  TRO is contained in $M_{n_k,m_k}$.   Conversely recall that 
 products $a_1 a_2^* \cdots a_n$ are dense in $B$ for $a_k \in A$, so that products $$p_k a_1 a_2^* \cdots a_n q_k = (p_k a_1 q_k) (q_k a_2^*p_k)  \cdots (p_k a_n q_k)$$ 
 are dense in $M_{n_k,m_k}$.  Now $I(A_k) = {\mathcal T}(A_k)$ is a quotient TRO of $M_{n_k,m_k}$ by a TRO ideal $J$. 
   In fact this TRO ideal is $(0)$, since $M_{n_k,m_k}$ has no nontrivial TRO ideals. 
   Thus $I(A_k) = M_{n_k,m_k}$ completely isometrically. 

 Suppose that $z , w \in {\rm Ball}(Z)$ with  $ab = az^*b = b w^* a$ for $a, b \in A$.   Then 
   $$p_k a b q_k = p_k a q_k p_k b q_k =  p_k a q_k (p_k q_k)^*  p_k a q_k.$$   It follows that 
  the element $z$ from Theorem \ref{KP} corresponding to the product in $A_k$ is $z_k = p_k q_k$.   The meaning of the last product can be interpreted as a product in $M_\ell$ where $\ell = \max\{ n_k, m_k \}$.   Indeed $p_k$ and $q_k$ are simply identity matrices of different sizes,
  so their product is the smaller of the two, thought of as inside $M_{n_k,m_k}$ (either as $[ I : 0]$ or its transpose).
    
  Claim:  $A_k$ is reversible.      Indeed 
$$(p_k b q_k )(p_k a q_k) = p_k b  a q_k = p_k a w^* b q_k = p_k a q_k w^* p_k b q_k 
= (p_k a q_k ) (p_k w q_k)^*(p_k b q_k),$$
for $a, b \in A.$
Since this equals $(p_k a q_k ) (p_k w q_k)^*(p_k a q_k)$ it follows from Theorem \ref{BRSKP}  that $A_k$ is reversible.   

   It is easy to see for each $k$ that 
  either $a = a z_k^*$ for all $a \in A_k$ or $a = z^*_k a$ for all $a \in A_k$.   Since $z^*_k A_k$ and $A_k z^*_k$ are commutative
  it follows that $A_k$ is commutative.  
  Then 
 $$a_1 a_2 = (\sum_{k=1}^r \, p_k) a_1 a_2  (\sum_{k=1}^r \, q_k) = \sum_{k=1}^r \, (p_k a_1 q_k )(p_k a_2  q_k), \qquad a_1, a_2 \in A.$$
By the last line we may reverse the parentheses and argue backwards to obtain $a_2 a_1$.   So $A$ is commutative. 
\end{proof} 

{\bf Remark.} Example \ref{Ex1} shows that some such conditions as those in the last results seem to be needed.  Note that one may also use the {\em lower right} corner in $M_k$ for $k$ large enough (as opposed to, or in addition to,
 `top left') in the definition of `ternary matrix algebra', and Corollary  \ref{recta} will still be valid.
 However this form  is equivalent (by conjugating by a unitary). 
 
\bigskip
 
 {\em Closing remark.}  The classical program of classifying 
 (real or complex) algebras foundered in the 19th century already.   There is no very useful complete 
classification  even of commutative finite dimensional algebras, or equivalently of 
commutative  subalgebras of matrix algebras. 
This is in part connected with the dramatic issues discovered by Gelfand and Ponomarev
\cite{GP}, and expanded on by very many other mathematicians since then, which show that there is no hope 
for classifying certain matrix structures (such problems are called `wild').   
For example the problem of classifying pairs of matrices up to simultaneous similarity contains the problem of classifying $k$-tuples of matrices up to simultaneous similarity for an arbitrary $k$, and it thus contains the problem of classifying representations of an arbitrary $k$ dimensional algebras.  At this point it is not completely clear, but it seems very likely,  that further progress on the classification of reversible finite dimensional operator algebras will be impacted in places by these difficulties.

 \subsection*{Acknowledgements}   
 We acknowledge support from NSF Grant DMS-2154903.  We thank Tao Mei for several conversations and helpful questions, and we thank David Sherman, Eusebio Gardella, Caleb McClure, Mehrdad Kalantar, and Bill Johnson  for other discussions, and the referee for their comments.  We also thank Caleb McClure for  numerically checking the dimensions of some generated $C^*$-algebras of some finite dimensional examples, and finding some typos.    We thank Deguang Han for the invitation (which came as we were writing this paper) to be included in this issue honoring Professor David Larson, and to him and Liming Ge for several editorial exchanges. We recently enjoyed a couple of excellent dinners and conversations on these topics with David Larson, and we thank Tao Mei for hosting these.  We also thank Tao  for his great hospitality, hosting the author for several days at Baylor University where this project started to get off the ground.      Indeed this project initially took shape at that time: the author suggested to Tao the classification of very low dimensional reversible operator algebras, as a great REU project for a student.

\end{document}